\documentclass[12pt]{amsart}
\usepackage{amssymb}
\usepackage{amsfonts}
\usepackage{amssymb,latexsym}
\usepackage{enumerate}
\usepackage{mathrsfs}

\makeatletter
\@namedef{subjclassname@2010}{%
  \textup{2010} Mathematics Subject Classification}
\makeatother

  [2002/01/19 v2.2g %
    AMS font definitions%
  ]
\DeclareFontFamily{U}{euf}{}
\DeclareFontShape{U}{euf}{m}{n}{%
  <5><6><7><8><9>gen*eufm%
  <10><10.95><12><14.4><17.28><20.74><24.88>eufm10%
  }{}
\DeclareFontShape{U}{euf}{b}{n}{%
  <5><6><7><8><9>gen*eufb%
  <10><10.95><12><14.4><17.28><20.74><24.88>eufb10%
  }{}

  [2002/01/19 v2.2g %
    AMS font definitions%
  ]
\DeclareFontFamily{U}{msb}{}
\DeclareFontShape{U}{msb}{m}{n}{%
  <5><6><7><8><9>gen*msbm%
  <10><10.95><12><14.4><17.28><20.74><24.88>msbm10%
  }{}

  [2002/01/19 v2.2g %
    AMS font definitions%
  ]
\DeclareFontFamily{U}{msa}{}
\DeclareFontShape{U}{msa}{m}{n}{%
  <5><6><7><8><9>gen*msam%
  <10><10.95><12><14.4><17.28><20.74><24.88>msam10%
  }{}

\newtheorem{theorem}{Theorem}[section]
\newtheorem{lemma}[theorem]{Lemma}
\newtheorem{proposition}[theorem]{Proposition}

\newtheorem{corollary}[theorem]{Corollary}

\theoremstyle{definition}
\newtheorem{remark}[theorem]{Remark}

\numberwithin{equation}{section} 

\begin{document}

\title[]
{Special values and integral representations for the Hurwitz-type Euler zeta functions}

\author{Su Hu}
\address{Department of Mathematics, South China University of Technology, Guangzhou, Guangdong 510640, China}
\email{mahusu@scut.edu.cn}

\author{Daeyeoul Kim}
\address{Department of Mathematics and Institute of Pure and Applied Mathematics, Chonbuk National University, 567 Baekje-daero, deokjin-gu, Jeonju-si, Jeollabuk-do 54896, Korea}
\email{kdaeyeoul@jbnu.ac.kr}

\author{Min-Soo Kim}
\address{Division of Mathematics, Science, and Computers, Kyungnam University, 7(Woryeong-dong) kyungnamdaehak-ro, Masanhappo-gu, Changwon-si,
Gyeongsangnam-do 51767, Korea}
\email{mskim@kyungnam.ac.kr}


\begin{abstract}
The Hurwitz-type Euler zeta function is defined as a deformation of the Hurwitz zeta function:
\begin{equation*}
\zeta_E(s,x)=\sum_{n=0}^\infty\frac{(-1)^n}{(n+x)^s}.
\end{equation*}
In this paper, by using the method of Fourier expansions, we shall evaluate  several integrals with integrands involving
Hurwitz-type Euler zeta functions $\zeta_E(s,x)$. Furthermore, the relations between the values of
a class of the Hurwitz-type (or Lerch-type) Euler zeta functions at rational arguments have also been given.

\end{abstract}



\maketitle




\section{Introduction}
\label{Intro}

The Hurwitz zeta function is defined by
\begin{equation}~\label{Hurwitz}
\zeta(s,x)=\sum_{n=0}^{\infty}\frac{1}{(n+x)^{s}}
\end{equation}
for $s\in\mathbb{C}$ and $x\neq0,-1,-2,\ldots,$ and the series converges for Re$(s)>1,$ so that $\zeta(s,x)$ is an analytic
function of $s$ in this region.
Setting $x=1$ in (\ref{Hurwitz}), it reduces to the Riemann zeta function
\begin{equation}~\label{o-Riemann}
\zeta(s)=\sum_{n=1}^{\infty}\frac{1}{n^{s}}.
\end{equation}
From the well-known identity
$\zeta(s,\frac12)=(2^s-1)\zeta(s),$ we see that the only real zeroes of $\zeta(s,x)$ with $x=\frac12$ are $s=0,-2,-4, \ldots.$

Its special values at nonpositive integers are Bernoulli polynomials, that is,
\begin{equation}~\label{H-Sp}
\zeta(1-m,x)=-\frac{1}{m}B_{m}(x), \quad m\in\mathbb N
\end{equation}
(see \cite[p. 162, (1.19)]{EM}).
Here the Bernoulli polynomials $B_{m}(x)$ are defined by their generating function
\begin{equation}~\label{def-ep}
\frac{ze^{xz}}{e^{z}-1}=\sum_{m=0}^{\infty}B_{m}(x)\frac{z^{m}}{m!},\quad |z| < \pi.
\end{equation}
Bernoulli polynomials have many interesting properties and arise in various areas of
mathematics (see \cite{Ber,SC,SO}).
Setting  $x=0$ in (\ref{def-ep}), we get the Bernoulli numbers
$$B_m=B_m(0),\quad m\in\mathbb N_0=\mathbb N\cup\{0\}.$$
From an easy manipulations of the generating function (\ref{def-ep}), we see that $B_m(1)=B_m~(m\neq1)$ and $ B_1(1)=-B_1,$
where $m\in\mathbb N_0$ (see \cite[p. 1480, (2.3)]{DV} and \cite[p.~529, (17)]{SC}).

The Hurwitz zeta function is a fundamental tool for  studying Stark's conjectures of algebraic number fields,
since it represents the partial zeta function of cyclotomic field (see e.g., \cite[p. 993, (4.2)]{Gross} or \cite[p.~4248]{HK-G}).
It also plays a role in the evaluation of fundamental determinants that appears in mathematical physics
(see \cite{Eliz} or \cite[p. 161, (1.14) and (1.15)]{EM}).

In this paper, we consider the Hurwitz-type Euler zeta function, which  is defined as a deformation of the Hurwitz zeta function
\begin{equation}\label{HEZ}
\zeta_E(s,x)=\sum_{n=0}^\infty\frac{(-1)^n}{(n+x)^s},
\end{equation}
for $s\in\mathbb{C}$ and $x\neq0,-1,-2,\ldots.$ This series converges for Re$(s)>0$, and it  can be analytically
continued to the complex plane without any pole.
It also represents a partial zeta function of cyclotomic fields in one version of Stark's conjectures in algebraic number theory (see \cite[p. 4249, (6.13)]{HK-G}),
and its special case, the alternative series,
\begin{equation}~\label{Riemann}
\zeta_{E}(s)=\sum_{n=1}^{\infty}\frac{(-1)^{n-1}}{n^{s}},
\end{equation}
is also a particular case of Witten's zeta functions in mathematical physics.  (See \cite[p. 248, (3.14)]{Min}).
We here mention that several  properties of $\zeta_E(s)$ can be found  in \cite{AS,Ay,CS}.
For example, in the form on \cite[p.~811]{AS}, the left hand side is the special values of the Riemann zeta functions at positive integers,
and  the right hand side is the special values of Euler zeta functions at positive integers.

Using the Fourier expansion formula
\begin{equation}\label{H-F-Exp}
\zeta(s,x)=\frac{2\Gamma(1-s)}{(2\pi)^{1-s}}\sum_{n=1}^\infty\frac{\sin\left(2n\pi x+\frac{\pi s}{2}\right)}{n^{1-s}},
\end{equation}
Espinosa and Moll \cite{EM} evaluated several fundamental integrals with integrands involving  Hurwitz zeta functions,
where $\Gamma$ is the Euler gamma-function (see, e.g., \cite{Mi} or \cite[Chapter 43]{SO}),
who studied ``the Hurwitz transform'' meaning an integral over $[0,1]$
of a Fourier series multiplied by the Hurwitz zeta function $\zeta(s,x),$ and obtained numerous results for those which aries
from the Hurwitz formula. Mez\H{o}~\cite{Mezo} determined the Fourier series expansion of the log-Barnes function, which is the analogue of the classical result of Kummer and Malmsten.
Applying this expansion, he also got some integrals similar to the Espinosa--Moll $\log$--Gamma integrals with respect to $\log G$ in \cite{EM,EM2}, where $G$ is Barnes $G$ function.

Recently, Shpot, Chaudhary and Paris~\cite{shpot} evaluated two integrals over $x\in [0,1]$ involving products of the function $\zeta_{1}(s,x)=\zeta(s,x)-x^{-a}$ for Re$(s) > 1$. As an application, they also calculated  the $O(g)$ weak-coupling expansion coefficient $c_{1}(\epsilon)$ of the Casimir energy for a film with Dirichlet-Neumann boundary conditions, first stated by Symanzik~\cite{symanzik} under the framework of $g\phi_{4-\epsilon}^{4}$ theory in quantum physics.

Here we will follow the approach of Espinosa and Moll in \cite{EM} to  evaluate  several integrals with integrands involving  Hurwitz-type Euler zeta functions $\zeta_E(s,x)$.

Our main tool is  the following Fourier expansion of $\zeta_E(s,x)$
\begin{equation}\label{F-Exp}
\zeta_E(s,x)=\frac{2\Gamma(1-s)}{\pi^{1-s}}\sum_{n=0}^\infty\frac{\sin\left((2n+1)\pi x+\frac{\pi s}{2}\right)}{(2n+1)^{1-s}}
\end{equation}
(comparing with (\ref{H-F-Exp}) above).
This expression, valid for ${\rm Re}(s)<1$ and $0<x\leq1,$ is derived by Williams and Zhang \cite[p.~36, (1.7)]{WY}.

The paper is organized as follows.

In Section \ref{f-e-ze}, we determine  the Fourier coefficients of $\zeta_E(s,x).$

In Section \ref{p-h-e-zeta}, we evaluate integrals with integrands consisting of products of two Hurwitz-type Euler zeta functions.
These will imply some classical relations for the integral of products of two Euler polynomials
(see Proposition \ref{ex-rem} and \ref{ex-rem-3}). We also evaluate the moments of the Hurwitz-type Euler zeta functions and
the Euler polynomials (see Proposition \ref{ex-rem-2} and \ref{ex-rem-3}).

In Section \ref{ex-ft}, we prove several further consequences of integrals with integrands consisting of products of the Hurwitz-type Euler zeta functions,
the exponential function, the power of sine and cosine.

In Section \ref{ex-cata}, we obtain an Euler-type formula for the Dirichlet beta function $\beta(2m)$ and
consider the Catalan's constant $G$ compiled by Adamchik in \cite[Entry 19]{Ad}.
From this, we get a finite closed-form expression for any $\beta(2m)$ in terms of the elementary functions (see \cite{Li}).

In Section \ref{v-le-rat}, we prove the relations between values of
a class of the Hurwitz-type (or Lerch-type) Euler zeta functions at rational arguments.

\section{Preliminaries}

The methods employed in this section may also be used to derive several results of the Hurwitz-type Euler zeta functions,
which are analogous of the known results for Hurwitz zeta functions.

The series expansion (\ref{HEZ}) for $\zeta_E(s,x)$ has a meaning if ${\rm Re}(s)>0.$ In the following, we recall some basic results of this functions given in \cite{WY}.

From (\ref{HEZ}), we can easily derive the following basic properties for $\zeta_E(s,x)$
\begin{equation}\label{HZ-inf-b}
\zeta_E(s,x)+\zeta_E(s,x+1)=x^{-s},
\end{equation}
\begin{equation}\label{HZ-inf-m}
\zeta_E(s,kx)=k^{-s}\sum_{n=0}^{k-1}(-1)^n \zeta_{E}\left(s, \frac nk +x\right)\quad\text{for $k$ odd},
\end{equation}
\begin{equation}\label{HZ-inf}
\zeta_E(s,x)-x^{-s}=-\sum_{n=0}^\infty\binom{-s}{n}\zeta_E(s+n)x^n
\end{equation}
(see \cite[(5), (6) and (7)]{Ap}).
For fixed $s\neq 1,$ the series in (\ref{HZ-inf}) converges absolutely for $|x|<1.$
For Re$(s)>1,$ (\ref{o-Riemann}) and (\ref{Riemann}) are connected by the following equation
\begin{equation}\label{ez-HZ-r}
\zeta_E(s)=(1-2^{1-s})\zeta(s).
\end{equation}

In \cite[\S 3]{WY}, Williams and Zhang established
the following integral representation
\begin{equation}\label{con-int-E}
\zeta_E(s,x)=\frac{e^{-\pi is}\Gamma(1-s)}{2\pi i}
\int_C\frac{e^{(1-x) z}z^{s-1}}{e^{z}+1}dz,
\end{equation}
where $C$ is the contour consisting of the real axis from $+\infty$ to $\epsilon,$ the circle $|z|=\epsilon$ and the real axis from $\epsilon$ to
$+\infty.$
 By (\ref{con-int-E}) and the generating function of the Euler  polynomials
\begin{equation}\label{Eu-pol}
\frac{2e^{xz}}{e^z+1}=\sum_{n=0}^\infty E_n(x)\frac{z^n}{n!}, \quad |z|<\pi ,
\end{equation}
we have
$$\zeta_E(-m,x)=\frac{(-1)^mm!}{4\pi i}
\int_{|z|=\epsilon}\sum_{n=0}^\infty\frac{E_n(1-x)}{n!}\frac{z^n}{z^{m+1}}dz,$$
that is,
\begin{equation}\label{EZ-Eu}
\zeta_E(-m,x)=\frac{(-1)^m}{2}E_m(1-x)=\frac12 E_m(x),\quad m\in\mathbb N_0
\end{equation}
(comparing with (\ref{H-Sp}) above).
 Setting $s=-m$ in (\ref{ez-HZ-r}), from (\ref{H-Sp}) and (\ref{EZ-Eu})  it is easy to see that
\begin{equation}\label{eu-be}
E_m(0)=-E_m(1)=\frac2{m+1}\left(1-2^{m+1}\right)B_{m+1}
\end{equation}
(see \cite[p.~529, (16)]{SC}).

The above expansion (\ref{EZ-Eu}) has also been derived in \cite[p.~41, (3.8)]{WY}.

The integers
$$E_{m}=2^{m}E_{m}\left(\frac{1}{2}\right),\quad m\in\mathbb N_0,$$
are called $m$-th Euler numbers.
For example, $E_0=1,E_2=-1,E_4=5,$ and $E_6=-61.$
Notice that the Euler numbers with odd subscripts vanish, that is, $E_{2m+1}=0$ for all $m\in\mathbb N_0.$
The Euler polynomials can be expressed in terms of the Euler numbers in the following way:
\begin{equation}\label{E-nu-pol-re}
E_m(x)=\sum_{i=0}^m\binom mi \frac{E_i}{2^i}\left(x-\frac12\right)^{m-i}
\end{equation}
(see \cite[p.~531, (29)]{SC}).

Some properties of Euler polynomials can be easily derived from their generating functions,
for example, from (\ref{Eu-pol}), we have
\begin{equation}\label{id-1}
E_m(x+1)+E_m(x)=2x^m,\quad m\in\mathbb N_0,
\end{equation}
which have numerous applications (see \cite[p. 1489, (5.4)]{DV} and \cite[p.~530, (23) and (24)]{SC}).

For further   properties of the Euler polynomials and numbers including their applications, we refer to
\cite{AS,Ba,GR,KH,SC}.
It may be interesting to point out that there is also a connection between the generalized Euler numbers and the ideal class group of the
$p^{n+1}$-th cyclotomic field when $p$ is a prime number.
For details, we refer to a recent paper~\cite{HK-I}, especially~\cite[Proposition 3.4]{HK-I}.

The result
\begin{equation}\label{e-1}
\int_0^1\zeta_E(s,x)dx=\frac{4\Gamma(1-s)}{\pi^{2-s}}\cos\left(\frac{\pi s}{2}\right)\lambda(2-s),
\end{equation}
valid for ${\rm Re}(s)<1,$ follows directly from the representation (\ref{F-Exp}).
Here $\lambda(s)$ is the Dirichlet lambda function
\begin{equation}\label{lam-def}
\lambda(s)=\sum_{n=0}^\infty\frac1{(2n+1)^s}=(1-2^{-s})\zeta(s),\quad {\rm Re}(s)>1,
\end{equation}
where $\zeta(s)$ is the Riemann zeta function (see \cite[p. 807, 23.2.20]{AS} and \cite[p. 32, Eq.~3:6:2]{SO}).
If we put $s=-2m$  with $m\in\mathbb N_0$ in (\ref{e-1})
and use (\ref{EZ-Eu}), then we obtain
\begin{equation}\label{lam-def-sp}
\int_0^1E_{2m}(x)dx=\frac{8(2m)!(-1)^m}{\pi^{2m+2}}\lambda(2m+2).
\end{equation}
If we replace $m$ by $2m$ and let $n=0$ in (\ref{ex-5}) below, then the left hand side of (\ref{lam-def-sp}) becomes
\begin{equation}\label{lam-def-sp-2}
\int_0^1E_{2m}(x)dx=-\frac{2(2m)!}{(2m+1)!}E_{2m+1}(0),
\end{equation}
where $m\in\mathbb N_0.$
From (\ref{lam-def-sp}) and (\ref{lam-def-sp-2}), the values of the Dirichlet lambda function
at even positive integers is given as
\begin{equation}\label{fa-fo}
\lambda(2m)=(-1)^m\frac{\pi^{2m}}{4(2m-1)!}E_{2m-1}(0),
\end{equation}
where $m\in\mathbb N.$
In particular, if we put $m=1,m=2$ and $m=3$ in (\ref{fa-fo}), then we have the first few values
$\lambda(2)=\frac{\pi^2}8,\lambda(4)=\frac{\pi^4}{96}$ and $\lambda(6)=\frac{\pi^6}{960}$ (see \cite[p.~808]{AS}).

\section{The Fourier expansion of $\zeta_E(s,x)$}\label{f-e-ze}

In this section, by using the Fourier expansion (\ref{F-Exp}) for $\zeta_E(s,x),$ we shall evaluate definite
integrals of the Hurwitz-type transform
$$\mathfrak{H}(f,s)=\int_0^1f(x)\zeta_E(s,x)dx.$$
The expansion is valid for ${\rm Re}(s)<1.$
For the basics we refer the reader to definite integrals of the Hurwitz transform, e.g., \cite{EM} and \cite{Li09}.

The following is the Fourier coefficients of $\zeta_E(s,x),$ as a direct consequence of (\ref{F-Exp})
(see e.g., \cite[p.~164, Proposition 2.1]{EM}).

\begin{proposition}\label{lem1}
The Fourier coefficients of $\zeta_E(s,x)$ are given by
\begin{equation}\label{sin-ET}
\int_0^1\sin((2k+1)\pi x)\zeta_E(s,x)dx=\frac{\pi^s(2k+1)^{s-1}}{2\Gamma(s)}\csc\left(\frac{\pi s}{2}\right)
\end{equation}
and
\begin{equation}\label{cos-ET}
\int_0^1\cos((2k+1)\pi x)\zeta_E(s,x)dx=\frac{\pi^s(2k+1)^{s-1}}{2\Gamma(s)}\sec\left(\frac{\pi s}{2}\right).
\end{equation}
\end{proposition}

\begin{remark}Recently, Luo~\cite{Luof}, Bayad~\cite{Ba2}, Navas, Francisco and
Varona~\cite{NFV} investigated Fourier expansions of the
Apostol-Bernoulli  polynomials, which are special values of Hurwitz-Lerch zeta function at non-positive integers~\cite[Lemma 2.1]{KH2}.
\end{remark}

\begin{proof}[Proof of Proposition \ref{lem1}]
The proof follows from the similar argument with \cite[Proposition 2.1]{EM}, by applying the orthogonality of the trigonometric functions
and the expansion (\ref{F-Exp}) above.
\end{proof}

In the following proposition, we shall compute the Hurwitz-type transform of powers of sine and cosine, including some related integrals
(see \cite[\S11]{EM}).

\begin{proposition}
Let $s\in\mathbb R_0^-=(-\infty,0]$ and $n\in\mathbb N.$ Then we have
\begin{equation}\label{m-sin-ET}
\begin{aligned}
\int_0^1\sin^{2n-1}(\pi x)\zeta_E(s,x)dx=\frac{\Gamma(1-s)}{\pi^{1-s}2^{2n-2}}\cos\left(\frac{\pi s}{2}\right)
\sum_{k=0}^{n-1}\frac{(-1)^k}{(2k+1)^{1-s}}\binom{2n-1}{n-k-1}
\end{aligned}
\end{equation}
and
\begin{equation}\label{m-cos-ET}
\begin{aligned}
\int_0^1\cos^{2n-1}(\pi x)\zeta_E(s,x)dx =\frac{\Gamma(1-s)}{\pi^{1-s}2^{2n-2}}\sin\left(\frac{\pi s}{2}\right)
\sum_{k=0}^{n-1}\frac{1}{(2k+1)^{1-s}}\binom{2n-1}{n-k-1}.
\end{aligned}
\end{equation}
\end{proposition}
\begin{proof}
The proofs follow from the similar arguments with \cite[Examples 11.1 and 11.2]{EM}, which is based on (\ref{sin-ET}) and (\ref{cos-ET}), and the integral representation
\begin{equation}\label{tri-int-sin}
\sin^{2n-1}(\pi x)=\frac1{2^{2n-2}}\sum_{k=0}^{n-1}(-1)^k\binom{2n-1}{n-k-1}\sin((2k+1)\pi x)
\end{equation}
by Kogan in \cite[p. 31, 1.320]{GR}.
\end{proof}

As in Espinosa and Moll~\cite{EM}, we may evaluate the integrals using the Fourier expansions of functions.
The following result about this idea will become the basic tool for the proofs of the results in this paper.

\begin{theorem}\label{thm-fu-co}
Let $f(w,x)$ be defined for $x\in[0,1]$ and a parameter $w.$ Let
\begin{equation}\label{thm-f-e}
f(w,x)=\sum_{n=0}^\infty\biggl(a_n(w)\cos((2n+1)\pi x)+b_n(w)\sin((2n+1)\pi x)\biggl)
\end{equation}
be its Fourier expansion, so that
\begin{equation}\label{thm-f-c 1}
a_n(w)=2\int_0^1 f(w,x)\cos((2n+1)\pi x)dx,
\end{equation}
\begin{equation}\label{thm-f-c 2}
b_n(w)=2\int_0^1 f(w,x)\sin((2n+1)\pi x)dx.
\end{equation}
Then, for $s\in\mathbb R_0^-,$ we have
\begin{equation}\label{thm-f 1}
\int_0^1 f(w,x)\zeta_E(s,x)dx=\frac{\Gamma(1-s)}{\pi^{1-s}}
\left(\sin\left(\frac{\pi s}{2}\right)C(s,w)
+\cos\left(\frac{\pi s}{2}\right)S(s,w) \right)
\end{equation}
and
\begin{equation}\label{thm-f 2}
\int_0^1 f(w,x)\zeta_E(s,1-x)dx=\frac{\Gamma(1-s)}{\pi^{1-s}}
\left(\cos\left(\frac{\pi s}{2}\right)S(s,w)-\sin\left(\frac{\pi s}{2}\right)C(s,w) \right),
\end{equation}
where
$$
C(s,w)=\sum_{n=0}^\infty\frac{a_n(w)}{(2n+1)^{1-s}}\quad\text{and}\quad S(s,w)=\sum_{n=0}^\infty\frac{b_n(w)}{(2n+1)^{1-s}}.
$$
\end{theorem}
\begin{proof}
The proof follows from the similar argument with \cite[Theorem 2.1]{EM}, by multiplying the formula (\ref{thm-f-e}) by $\zeta_E(s,x)$ and integatationg
both sides from 0 to 1, then applying (\ref{sin-ET}) and (\ref{cos-ET}).
\end{proof}

\section{Integral representations of Hurwitz-type Euler zeta functions}\label{p-h-e-zeta}

In this section, we shall evaluate integrals with integrands consisting of products of Hurwitz-type Euler zeta functions.
These will imply some integral representations for the function and classical relations for the integral of products of two Euler polynomials.

\begin{theorem}\label{lem-fu-co-2}
Let $s,s'\in\mathbb R_0^-.$ Then we have
\begin{equation}\label{thm-re}
\int_0^1\zeta_E(s',x)\zeta_E(s,x)dx=\frac{2\Gamma(1-s)\Gamma(1-s')}{\pi^{2-s-s'}}\lambda(2-s-s')\cos\left(\frac{\pi(s-s')}{2}\right).
\end{equation}
Similarly, we have
\begin{equation}\label{thm-re-2}
\int_0^1\zeta_E(s',x)\zeta_E(s,1-x)dx=\frac{2\Gamma(1-s)\Gamma(1-s')}{\pi^{2-s-s'}}\lambda(2-s-s')\cos\left(\frac{\pi(s+s')}{2}\right).
\end{equation}
\end{theorem}
\begin{proof}
In view of (\ref{thm-fu-co}), (\ref{thm-f-c 1}) and (\ref{thm-f-c 2}),
the expansion (\ref{F-Exp}) shows that the coefficients of $\zeta_E(s',x)$ are given by
\begin{equation}\label{coeffi}
\begin{aligned}
a_n(s')&=\frac{2\Gamma(1-s')\sin\left(\frac{\pi s'}2\right)}{\pi^{1-s'}}\frac1{(2n+1)^{1-s'}}, \\
b_n(s')&=\frac{2\Gamma(1-s')\cos\left(\frac{\pi s'}2\right)}{\pi^{1-s'}}\frac1{(2n+1)^{1-s'}}.
\end{aligned}
\end{equation}
From (\ref{lam-def}), (\ref{thm-f 1}) and (\ref{coeffi}), we obtain
\begin{equation}\label{coeffi-1}
\begin{aligned}
\int_0^1\zeta_E(s',x)\zeta_E(s,x)dx
&=\frac{2\Gamma(1-s)\Gamma(1-s')}{\pi^{2-s-s'}}
\biggl(\sin\left(\frac{\pi s}2\right)\sin\left(\frac{\pi s'}2\right) \\
&\quad+\cos\left(\frac{\pi s}2\right)\cos\left(\frac{\pi s'}2\right) \biggl)
\sum_{n=0}^\infty\frac{1}{(2n+1)^{2-s-s'}} \\
&=\frac{2\Gamma(1-s)\Gamma(1-s')}{\pi^{2-s-s'}}\cos\left(\frac{\pi(s-s')}{2}\right)\lambda(2-s-s'),
\end{aligned}
\end{equation}
which gives (\ref{thm-re}). The proof of (\ref{thm-re-2}) is similar.
\end{proof}

\begin{corollary}
Let $s,s'\in\mathbb R_0^-$ and
$$\delta_2(s)=\frac{1-2^{s+1}}{1-2^s}, \quad s\neq0.$$
Then we have
\begin{equation}\label{co-re}
\begin{aligned}
\int_0^1\zeta_E(s',x)\zeta_E(s,x)dx&=\delta_2(1-s-s')\frac{\cos\left(\frac{\pi(s-s')}{2}\right)}{\cos\left(\frac{\pi(s+s')}{2}\right)}
B(1-s,1-s')\\
&\quad\times\lambda(s+s'-1).
\end{aligned}
\end{equation}
Similarly, we have
\begin{equation}\label{co-re-2}
\begin{aligned}
\int_0^1\zeta_E(s',x)\zeta_E(s,1-x)dx&=\delta_2(1-s-s')B(1-s,1-s') \\
&\quad\times\lambda(s+s'-1).
\end{aligned}
\end{equation}
Here $B(s,s')$ is the beta function,
\begin{equation}\label{be-def}
B(s,s')=\frac{\Gamma(s)\Gamma(s')}{\Gamma(s+s')},
\end{equation}
defined for $s,s'\in\mathbb C$ with ${\rm Re}(s),{\rm Re}(s')>1,$ see \cite[(1.36)]{WY}.
\end{corollary}
\begin{proof}
The functional equation of Riemann's zeta function is well-known, and it is given by (see \cite[(1.35)]{EM})
\begin{equation}\label{ri-ze-ft}
\zeta(1-s)=2\cos\left(\frac{\pi s}{2}\right)\frac{\Gamma(s)}{(2\pi)^s}\zeta(s).
\end{equation}
From (\ref{lam-def}) and (\ref{ri-ze-ft}), we have
\begin{equation}\label{la-ft}
\lambda(2-s-s')=-\frac{\delta_2(1-s-s')\pi^{2-s-s'}}{2\Gamma(2-s-s')\cos\left(\frac{\pi (2-s-s')}{2}\right)}\lambda(s+s'-1).
\end{equation}
Finally, using (\ref{thm-re}) and (\ref{la-ft}), we obtain (\ref{co-re}).
The proof of (\ref{co-re-2}) is similar.
\end{proof}

\begin{corollary}\label{pro-fu-co-1}
Let $s\in\mathbb R_0^-.$ Then we have
\begin{equation}\label{ex-1-1}
\int_0^1\zeta_E^2(s,x)dx=2\Gamma^2(1-s)\pi^{2s-2}\lambda(2-2s)
\end{equation}
and
\begin{equation}\label{ex-1-2}
\int_0^1\zeta_E(s,x)\zeta_E(s,1-x)dx=2\Gamma^2(1-s)\pi^{2s-2}\lambda(2-2s)\cos(\pi s).
\end{equation}
\end{corollary}
\begin{proof}
Put $s=s'$ in (\ref{thm-re}) and (\ref{thm-re-2}), we get our results.
\end{proof}

\begin{corollary}\label{china-1}
Let $m\in\mathbb N.$ Then we have
\begin{equation}\label{ex-3}
\int_0^1\zeta_E^2\left(-m+\frac12,x\right)dx=2\left(\frac{(2m)!}{2^{2m}m!}\right)^2\frac{\lambda(2m+1)}{\pi^{2m}}.
\end{equation}
\end{corollary}
\begin{proof}
Setting $s=-m+\frac 12$ in (\ref{ex-1-1}), then by
$$\Gamma\left(m+\frac12\right)=\frac{\sqrt \pi (2m)!}{2^{2m}m!}\quad\text{(see \cite[p.~167]{EM})},$$
we get our result.
\end{proof}

\begin{corollary}\label{china-2}
Let $s\in\mathbb R_0^-$ and $m\in\mathbb N.$ Then we have
\begin{equation}\label{ex-4}
\int_0^1 E_{m-1}(x)\zeta_E(s,x)dx=(-1)^{m+1}2\delta_2(m-s)\frac{(m-1)!\lambda(s-m)}{(1-s)_m},
\end{equation}
where $(s)_k=s(s+1)\cdots(s+k-1)$ is the Pochhammer symbol.
\end{corollary}

\begin{remark}
The case $m=1$ in (\ref{ex-4}) implies (\ref{e-1}) from (\ref{lam-def}) and (\ref{ri-ze-ft}).
\end{remark}

\begin{proof}[Proof of the Corollary \ref{china-2}]
Let $s'=1-m$ in (\ref{co-re}), we have
$$\begin{aligned}
\int_0^1 \zeta_E(1-m,x)\zeta_E(s,x)dx&=\delta_2(m-s)\frac{\cos\left(\frac{\pi(s-1+m)}{2}\right)}
{\cos\left(\frac{\pi(s+1-m)}{2}\right)}B(1-s,m) \\
&\quad\times\lambda(s-m).
\end{aligned}$$
The result then follows from (\ref{EZ-Eu}), since
$$\frac{\cos\left(\frac{\pi(s-1+m)}{2}\right)}
{\cos\left(\frac{\pi(s+1-m)}{2}\right)}=(-1)^{m+1}$$
and $B(1-s,m)=(m-1)!/(1-s)_m.$
\end{proof}

\begin{lemma}\label{co-ezeta}
For $m\in\mathbb N,$ we have
\begin{equation}\label{lam-1-m}
\lambda(1-m)=(-1)^{m+1}\frac{1}{2\delta_2(m-1)}E_{m-1}(0).
\end{equation}
\end{lemma}
\begin{proof}
This follows from (\ref{H-Sp}) with $x=1,$ (\ref{eu-be}) and (\ref{lam-def}).
\end{proof}

The following formula on the integral of two Euler polynomials has a long history, it has already appeared in a book by N\"orlund \cite[p.~36]{No}.
See e.g. \cite[p.~530, (25)]{SC} and \cite[p. 64, (52)]{Sri}.

\begin{corollary}\label{ex-rem}
Let $m,n\in\mathbb N_0.$ Then we have
\begin{equation}\label{ex-5}
\int_0^1 E_m(x)E_n(x)dx=2(-1)^{n+1}\frac{m!n!}{(m+n+1)!}E_{m+n+1}(0).
\end{equation}
\end{corollary}
\begin{proof}[Proof of Corollary \ref{ex-rem}]
Set $s=1-n$, where $n\in\mathbb N$ in (\ref{ex-4}), then by (\ref{EZ-Eu}) and (\ref{lam-1-m}), we obtain
$$
\begin{aligned}
\int_0^1 E_{m-1}(x)E_{n-1}(x)dx&=4(-1)^{m+1}\delta_2(m+n-1)\frac{(m-1)!\lambda(1-(m+n))}{(n)_m} \\
&=2(-1)^{n}\frac{(m-1)!}{(n)_m} E_{m+n-1}(0) \\
&=2(-1)^n\frac{(m-1)!(n-1)!}{(m+n-1)!}E_{m+n-1}(0),
\end{aligned}
$$
where $m,n\in\mathbb N.$ This completes our proof.
\end{proof}

\begin{corollary}\label{ex-rem-2}
For $n\in\mathbb N_0,$ the moments of the Hurwitz-type Euler zeta function are given by
\begin{equation}\label{mon-zeta}
\begin{aligned}
\int_0^1 x^n\zeta_E(s,x)dx&=\sum_{j=0}^n\binom nj(-1)^{j}\delta_2(j-s+1)\frac{j!\lambda(s-j-1)}{(1-s)_{j+1}} \\
&\quad+(-1)^{n}\delta_2(n-s+1)\frac{n!\lambda(s-n-1)}{(1-s)_{n+1}}.
\end{aligned}
\end{equation}
\end{corollary}
\begin{proof}
The proof of (\ref{mon-zeta}) is obtained by using the expansion of $x^n$ in terms of Euler polynomials (\ref{id-1}) and the evaluation (\ref{ex-4}).
\end{proof}

\begin{corollary}\label{ex-rem-3}
Let $m\in\mathbb N$ and $n\in\mathbb N_0.$ Then we have
\begin{equation}\label{ex-6}
\int_0^1x^nE_{m-1}(x)dx=\frac{(-1)^m}{m}\left(\sum_{j=0}^n\frac{\binom nj}{\binom{m+j}j}E_{m+j}(0)+\frac{E_{m+n}(0)}{\binom{m+n}n}\right).
\end{equation}
\end{corollary}
\begin{proof}
By (\ref{lam-1-m}), we have
\begin{equation}\label{lam-p}
\lambda(1-(j+m+1))=(-1)^{j+m}\frac{1}{2\delta_2(j+m)}E_{j+m}(0).
\end{equation}
Letting $s=1-m$ in (\ref{mon-zeta}), then using (\ref{EZ-Eu}) and (\ref{lam-p}), we easily deduce that
$$\int_0^1x^nE_{m-1}(x)dx=(-1)^m\left(\sum_{j=0}^n\binom nj\frac{j!}{(m)_{j+1}}E_{m+j}(0)+\frac{n!}{(m)_{n+1}}E_{m+n}(0)\right)$$
and this implies our result.
\end{proof}

\begin{remark}
Corollary \ref{ex-rem-3} gives some identities of $E_m(0)$ for each value of $m\in\mathbb N.$
For instance, when $m=1$ and 2, we have
$$\sum_{j=0}^n\binom nj\frac{E_{j+1}(0)}{j+1}=-\frac{1}{n+1}(E_{n+1}(0)+1),$$
$$\sum_{j=0}^n\binom nj\frac{E_{j+2}(0)}{(j+1)(j+2)}=-\frac{1}{2(n+1)(n+2)}(2E_{n+2}(0)-n),$$
where $n\in\mathbb N_0.$
\end{remark}

\section{The exponential function}\label{ex-ft}

In this section, we shall evaluate the Hurwitz-type transform of the exponential function (see \cite[\S4]{EM}).
As \cite[\S4]{EM}, the result is expressed in terms of the transcendental function

\begin{equation}\label{tran-ft}
F(x,s)=\sum_{n=0}^\infty\lambda(n+2-s)x^n\quad \text{for } |x|<1,
\end{equation}
where $\lambda(s)$ denotes the Dirichlet lambda function (define $0^0=1$ if $x=0$).

\begin{theorem}\label{ex-thm-1}
Let $s\in\mathbb R_0^-$ and $|t|<1.$  Then we have
\begin{equation}\label{thm-tran}
\int_0^1e^{2\pi tx}\zeta_E(s,x)dx=\frac{2(e^{2\pi t}+1)\Gamma(1-s)}{\pi^{2-s}}{\rm Re}\!\left(e^{\pi i s/2}F(2it,s) \right),
\end{equation}
where $F(x,s)$ is given in (\ref{tran-ft}).
\end{theorem}
\begin{remark}
Setting $t=0$ in (\ref{thm-tran}), we get (\ref{e-1}).
\end{remark}
\begin{proof}[Proof of Theorem \ref{ex-thm-1}]
The generating function for the Euler polynomials (\ref{Eu-pol}) gives
$$e^{xt}=\frac{e^t+1}{2}\sum_{n=0}^\infty E_n(x)\frac{t^n}{n!},$$
so that
\begin{equation}\label{ex-zeta}
\int_0^1e^{xt}\zeta_E(s,x)dx=\frac{e^t+1}{2}\sum_{n=0}^\infty \frac{t^n}{n!}\int_0^1E_n(x)\zeta_E(s,x)dx.
\end{equation}
By (\ref{be-def}), we have
\begin{equation}\label{beta-1}
\frac{n!}{(1-s)_{n+1}}=B(1-s,n+1)=\frac{\Gamma(1-s)\Gamma(n+1)}{\Gamma(n-s+2)}.
\end{equation}
Therefore, from (\ref{ex-4}) and (\ref{beta-1}), (\ref{ex-zeta}) gives
\begin{equation}\label{int1}
\begin{aligned}
\int_0^1e^{xt}\zeta_E(s,x)dx&=(e^t+1)\sum_{n=0}^\infty \frac{t^n}{n!}(-1)^n\delta_2(n-s+1) B(1-s,n+1) \\
&\quad\times \lambda(s-n-1).
\end{aligned}
\end{equation}
By (\ref{la-ft}), we have
\begin{equation}\label{la-ft-inv}
\lambda(s-n-1)=-\frac{2}{\delta_2(n-s+1)}\frac{\Gamma(n-s+2)\cos\left(\frac{\pi(n-s+2)}{2}\right)}{\pi^{n-s+2}}\lambda(n-s+2),
\end{equation}
thus in viewing of (\ref{beta-1}) and (\ref{la-ft-inv}), we my also write (\ref{int1}) in the form
\begin{equation}\label{re-im}
\begin{aligned}
\int_0^1e^{xt}\zeta_E(s,x)dx
&={(e^t+1)\Gamma(1-s)}\sum_{n=0}^\infty
\frac{t^n}{n!}(-1)^n\delta_2(n-s+1)\frac{\Gamma(n+1)}{\Gamma(n-s+2)} \\
&\quad\times \lambda(s-n-1) \\
&=\frac{2(e^t+1)\Gamma(1-s)}{\pi^{2-s}}\sum_{n=0}^\infty \left(\frac{t}{\pi}\right)^n(-1)^{n}
\cos\left(\frac{\pi(n-s)}{2}\right) \\
&\quad\times\lambda(n-s+2).
\end{aligned}
\end{equation}
Finally, by replacing $t$ with $2\pi t$ and using the identity
$$\cos\frac{\pi(n-s)}{2}= \text{Re}(e^{i\pi(s-n)/2}) = \text{Re}((-i )^n e^{i\pi s/2}),$$
we get
$$\begin{aligned}
\int_0^1e^{2\pi tx}&\zeta_E(s,x)dx \\
&=\frac{2(e^{2\pi t}+1)\Gamma(1-s)}{\pi^{2-s}}\sum_{n=0}^\infty \left(-2t\right)^n
\lambda(n-s+2)\cos\left(\frac{\pi(n-s)}{2}\right) \\
&=\frac{2(e^{2\pi t}+1)\Gamma(1-s)}{\pi^{2-s}}\sum_{n=0}^\infty \left(-2t\right)^n
\lambda(n-s+2)\, \text{Re}((-i )^n e^{i\pi s/2}) \\
&=\frac{2(e^{2\pi t}+1)\Gamma(1-s)}{\pi^{2-s}}\text{Re}\left(e^{\pi i s/2}\sum_{n=0}^\infty
\lambda(n-s+2) (2it )^n \right) \\
&=\frac{2(e^{2\pi t}+1)\Gamma(1-s)}{\pi^{2-s}}\text{Re}\left(e^{\pi i s/2}F(2it,s) \right). \\
\end{aligned}$$
This completes our proof.
\end{proof}

\begin{corollary}\label{ex-coro-2}
Let $m\in\mathbb N_0$ and $|t|<1,t\neq0.$  Then we have
\begin{equation}\label{cor-tran}
\begin{aligned}
\int_0^1e^{2\pi tx}E_m(x)dx&=\frac{(-1)^m4(e^{2\pi t}+1)m!}{(2\pi t)^{m+2}} \\
&\quad\times\left(
\frac{\pi t}{2}\tanh(\pi t) -\sum_{r=0}^{\left\lfloor \frac{m-1}2\right\rfloor}(-1)^r\lambda(2r+2)(2t)^{2r+2}
\right),
\end{aligned}
\end{equation}
where $\lfloor x\rfloor$ is the floor (greatest integer) function.
\end{corollary}
\begin{proof} Suppose that $k\in\mathbb N_0.$
Let $s=-2k$ in (\ref{thm-tran}). From (\ref{tran-ft}), we have
\begin{equation}\label{p-1}
\begin{aligned}
\text{Re}\!\left(e^{\pi i s/2}F(2it,s) \right)&=\text{Re}\!\left(e^{\pi i (-2k)/2}F(2it,-2k) \right) \\
&=(-1)^k \text{Re}\!\left(\sum_{r=0}^\infty\lambda(r+2+2k)(2it)^r\right) \\
&=(-1)^k \sum_{r=0}^\infty\lambda(2r+2+2k)(-1)^r(2t)^{2r}.
\end{aligned}
\end{equation}
From \cite[p. 37, Eq. 3:14:7]{SO}, we obtain
$$\tanh\left(\frac{\pi x}{2}\right)=-\frac4{\pi x}\sum_{r=1}^\infty(-1)^r\lambda(2r)x^{2r},$$
so
\begin{equation}\label{p-2}
\begin{aligned}
\frac{\pi x}4\tanh\left(\frac{\pi x}{2}\right)&=\sum_{r=0}^\infty(-1)^r\lambda (2r+2)x^{2r+2} \\
&=\sum_{r=k}^\infty(-1)^r\lambda (2r+2)x^{2r+2} +\sum_{r=0}^{k-1}(-1)^r\lambda (2r+2)x^{2r+2}.
\end{aligned}
\end{equation}
Replacing $x$ by $2t$ in (\ref{p-2}), we have
\begin{equation}\label{p-3}
\begin{aligned}
\sum_{r=k}^\infty(-1)^r\lambda (2r+2)(2t)^{2r+2} =\frac{\pi t}2\tanh(\pi t)-\sum_{r=0}^{k-1}(-1)^r\lambda (2r+2)(2t)^{2r+2}.
\end{aligned}
\end{equation}
On the other hand, by (\ref{p-1}), we have
\begin{equation}\label{p-4}
\begin{aligned}
(2t)^{-2k-2}\sum_{r=k}^\infty(-1)^r&\lambda (2r+2)(2t)^{2r+2} \\
&=(-1)^k\sum_{r=0}^\infty\lambda(2r+2+2k)(-1)^r(2t)^{2r}\\
&=\text{Re}\!\left(e^{\pi i (-2k)/2}F(2it,-2k) \right)
\end{aligned}
\end{equation}
for $|t|<1,t\neq0.$ By (\ref{p-3}) and (\ref{p-4}), we get
\begin{equation}\label{p-5}
\begin{aligned}
\text{Re}\!&\left(e^{\pi i (-2k)/2}F(2it,-2k) \right) \\
&=\frac{(-1)^{2k}}{(2t)^{2k+2}}\left(\frac{\pi t}2\tanh(\pi t)-\sum_{r=0}^{k-1}(-1)^r\lambda (2r+2)(2t)^{2r+2}\right).
\end{aligned}
\end{equation}

Let $s=-2k-1$ in (\ref{thm-tran}), similar with the above, we obtain
\begin{equation}\label{p-6}
\begin{aligned}
\text{Re}\!&\left(e^{\pi i s/2}F(2it,s) \right) \\
&=\text{Re}\!\left(e^{\pi i (-2k-1)/2}F(2it,-2k-1) \right) \\
&=\frac{(-1)^{2k+1}}{(2t)^{2k+3}}\left(\frac{\pi t}2\tanh(\pi t)-\sum_{r=0}^{k}(-1)^r\lambda (2r+2)(2t)^{2r+2}\right).
\end{aligned}
\end{equation}
The proof now follows directly from (\ref{EZ-Eu}), (\ref{thm-tran}), (\ref{p-5}) and (\ref{p-6}).
\end{proof}

\section{An expression for Catalan's constant and an Euler-type formula for the Dirichlet beta function}\label{ex-cata}

Following Berndt \cite[p. 200, Entry 17(v)]{Ber}, Espinosa and Moll \cite[p. 177, (8.1)]{EM}, for $x>0,$ we define
\begin{equation}\label{ca-1}
G_E(s,x)=\zeta_E(s,x)-\zeta_E(s,1-x).
\end{equation}
From (\ref{F-Exp}), the Fourier expansion for $\zeta_E(s,x),$ we have the following  Fourier expansion of $G_E(s,x)$
\begin{equation}\label{ca-four}
G_E(s,x)=\frac{4\Gamma(1-s)}{\pi^{1-s}}\sin\left(\frac{\pi s}{2}\right)\sum_{n=0}^\infty\frac{\cos((2n+1)\pi x)}{(2n+1)^{1-s}}.
\end{equation}

We introduce the anti-symmetrized Hurwitz transform for $G_E(s,x)$ as follows (also see \cite{EM})
\begin{equation}\label{am-sy Hu}
\mathfrak{H}_A(f)=\frac12\int_0^1 f(w,x)G_E(s,x)dx,
\end{equation}
where $f(w,x)$ is given in (\ref{thm-f-e}).
For a function $f(w,x)$ with Fourier expansion as in (\ref{thm-f-e}), we have
\begin{equation}\label{am-sy Hu}
\frac12\int_0^1 f(w,x)G_E(s,x)dx=
\frac{\Gamma(1-s)}{\pi^{1-s}}\sin\left(\frac{\pi s}{2}\right)\sum_{n=0}^\infty\frac{a_n(w)}{(2n+1)^{1-s}}.
\end{equation}
From \cite[p. 408, 3.612]{GR}, we get
\begin{equation}\label{cos01}
\int_0^1 \frac{\cos((2n+1)\pi x)}{\cos(\pi x)}dx=(-1)^n.
\end{equation}
By (\ref{ca-four}) and (\ref{cos01}), we see that
\begin{equation}\label{fu-cos01}
\frac12\int_0^1 \frac{G_E(s,x)}{\cos(\pi x)}dx=\frac{2\Gamma(1-s)}{\pi^{1-s}}\sin\left(\frac{\pi s}{2}\right)\beta(1-s),
\end{equation}
where
\begin{equation}\label{di-beta}
\beta(s)=\sum_{n=0}^\infty\frac{(-1)^n}{(2n+1)^s}
\end{equation}
is the Dirichlet beta function (see \cite[p. 807, 23.2.21]{AS}).
It is necessary to remark that by the procedure of analytic continuation, the function $\beta(s)$ can extend   to all points in the complex plane,
without any singularities.
In particular, when $s=-1,$ (\ref{fu-cos01}) becomes
\begin{equation}\label{fu-cos=-1}
\frac12\int_0^1 \frac{G_E(-1,x)}{\cos(\pi x)}dx=-\frac2{\pi^2}\beta(2).
\end{equation}
Espinosa and Moll \cite[p. 178, (8.8)]{EM} established the integral representation of the Catalan constant $G$ (see \cite{Ad}) as
\begin{equation}\label{ca-con-int}
\int_0^1\frac{\frac12 -x}{\cos(\pi x)}dx=\frac{4G}{\pi^2}.
\end{equation}
Using (\ref{EZ-Eu}), (\ref{ca-1}) with $s=-1$ and (\ref{ca-con-int}),
the left hand side of (\ref{fu-cos=-1}) becomes
\begin{equation}\label{ca-con}
\begin{aligned}
\frac12\int_0^1 \frac{G_E(-1,x)}{\cos(\pi x)}dx&=\frac12\int_0^1\frac{x-\frac12}{\cos(\pi x)}dx \\
&=-\frac{2G}{\pi^2}.
\end{aligned}
\end{equation}
From (\ref{fu-cos=-1}) and (\ref{ca-con}), we then have
\begin{equation}\label{ca-con-de}
G=\beta(2)=-\frac{\pi^2}{4}\int_0^1 \frac{G_E(-1,x)}{\cos(\pi x)}dx.
\end{equation}
This form of Catalan's constant is Entry 14 and 19 in the list of expressions for $G$ compiled by Adamchik in \cite{Ad}.

From (\ref{fu-cos01}), we have the following lemma, which were given in \cite[p. 807, 23.2.23]{AS}, \cite[p. 36, Eq. 3:13:4]{SO}
and \cite[(4.13)]{WY}.

\begin{lemma}\label{sec-eu}
Let $m\in\mathbb N.$ Then we have
\begin{equation}\label{ca-con-pro}
\int_0^1\sec(\pi x) E_{2m-1}(x)dx=(-1)^{m}\frac{4(2m-1)!}{\pi^{2m}}\beta(2m).
\end{equation}
\end{lemma}
\begin{proof}
For $m\in\mathbb N,$ letting $s=-(2m-1)$ in (\ref{fu-cos01}), we have
\begin{equation}\label{ca-pro-1}
\frac12\int_0^1 \frac{G_E(-(2m-1),x)}{\cos(\pi x)}dx=\frac{2(2m-1)!}{\pi^{2m}}(-1)^m\beta(2m).
\end{equation}
It is easy to see that by (\ref{Eu-pol}),
$$E_m(1-x)=(-1)^mE_{m}(x),\quad m\in\mathbb N_0.$$
Hence, by (\ref{EZ-Eu}) and (\ref{ca-1}), the left hand side of (\ref{ca-pro-1}) becomes
\begin{equation}\label{ca-pro-2}
\begin{aligned}
\frac12\int_0^1 \frac{G_E(-(2m-1),x)}{\cos(\pi x)}dx&=\frac14\int_0^1 \frac{E_{2m-1}(x)-E_{2m-1}(1-x)}{\cos(\pi x)}dx \\
&=\frac12\int_0^1 \frac{E_{2m-1}(x)}{\cos(\pi x)}dx
\end{aligned}
\end{equation}
for  $m\in\mathbb N.$
The proof now follows directly from (\ref{ca-pro-1}) and (\ref{ca-pro-2}).
\end{proof}

We now evaluate $\beta(2m)$ using (\ref{ex-6}). This new formula can be regarded as  an analogue of Theorem 1 of \cite{Li}.

\begin{theorem}[Euler-type formula for $\beta(2m)$]\label{sec-eu-cl}
Let $m\in\mathbb N.$ Then we have
\begin{equation}\label{di-bet-cl}
\beta(2m)=\sum_{n=1}^\infty\frac{(-1)^{n+m}\pi^{2m+2n}E_{2n}}{4}
\left(\sum_{j=1}^{n}\frac{E_{2m+2j-1}(0)}{(2n-2j+1)!(2m+2j-1)!}\right).
\end{equation}
\end{theorem}
\begin{proof}
From the Taylor expansion for $\sec(\pi x),$ namely
$$\sec(\pi x)=\sum_{n=0}^\infty\frac{(-1)^nE_{2n}}{(2n)!}(\pi x)^{2n},$$
it is clear that
$$\int_0^1\sec(\pi x) E_{2m-1}(x)dx=\sum_{n=0}^\infty\frac{(-1)^n\pi^{2n}E_{2n}}{(2n)!}\int_0^1 x^{2n} E_{2m-1}(x)dx,$$
where $m\in\mathbb N.$ Note that $E_{2m}(0)=0,m\in\mathbb N.$
Therefore, by (\ref{ex-6}) and (\ref{ca-con-pro}),  we get our result.
\end{proof}

For instance, if we set $m=1$ in (\ref{di-bet-cl}), then we obtain:

\begin{corollary}
$$\beta(2)=G=\sum_{n=1}^\infty\frac{(-1)^{n+1}n\pi^{2n+2}E_{2n}}{4(2n+2)!}.$$
\end{corollary}

\begin{remark}
For $m\in\mathbb N_0$ we write $s=-2m$ in (\ref{F-Exp}) and let $x=1/2.$ We obtain easily
\begin{equation}\label{beta-zeta-od}
\zeta_E\left(-2m,\frac12\right)=\frac{(-1)^m2(2m)!}{\pi^{2m+1}}\sum_{n=0}^\infty\frac{(-1)^n}{(2n+1)^{2m+1}},
\end{equation}
so (\ref{beta-zeta-od}) becomes, after using (\ref{EZ-Eu}),
\begin{equation}\label{beta-zeta-od-2}
\frac12 E_{2m}\left(\frac12\right)=\frac{(-1)^m2(2m)!}{\pi^{2m+1}}\sum_{n=0}^\infty\frac{(-1)^n}{(2n+1)^{2m+1}},
\end{equation}
By using $E_{2m}=2^{2m}E_{2m}\left({1}/{2}\right),m\in\mathbb N_0,$ we are easy to show that,
for the Dirichlet beta function (\ref{di-beta}),
\begin{equation}\label{beta-od}
\beta(2m+1)=(-1)^m\frac{E_{2m}}{2^{2m+2}(2m)!}\pi^{2m+1},
\end{equation}
where $E_m$ are the Euler number, that is, the integers obtained as the coefficients of $z^m/m!$
in the Taylor expansion of $1/\cosh z, |z|<\pi/2$ (see \cite[(3)]{Li} and \cite[(4.10)]{WY}).
The formula (\ref{ca-con-pro}) for $\beta(2m)$ gives us
\begin{equation}\label{beta-ev}
\beta(2m)=(-1)^{m}\frac{\pi^{2m}}{4(2m-1)!}\int_0^1\sec(\pi x) E_{2m-1}(x)dx,
\end{equation}
which is also Equation 23.2.23 of \cite{AS}, Equation 3:13:4 of \cite{SO} and (4.13) of \cite{WY}.
(\ref{beta-ev}) may give some hint for the question ``why it is so difficult to find closed-form expressions for even beta values?"
However, Lima \cite{Li} has derived the exact closed-form expression for a certain class of zeta series related to $\beta(2m)$ and a finite number
of odd zeta values.
\end{remark}

\section{Integral representations of Lerch-type Euler zeta functions and
special values for Lerch-type Euler zeta functions at rational arguments} \label{v-le-rat}

The counterpart of the Hurwitz-type Euler zeta function, the Lerch-type Euler zeta function
$\ell_{E,s}(x)$ of the complex exponential argument is defined by
\begin{equation}\label{le-eu-zeta}
\ell_{E,s}(x)=\sum_{n=0}^\infty \frac{e^{(2n+1)\pi i x}}{(2n+1)^s}
\end{equation}
for Re$(s)>1$ and $x\in\mathbb R$ (see \cite[p. 1530, (7)]{CK}).
The Hurwitz-type Euler zeta function and the Lerch-type Euler zeta function are related by
the functional equation
\begin{equation}\label{hu-le-zeta}
\zeta_{E}(1-s,x)=\frac{\Gamma(s)}{\pi^s}\left(e^{-\frac{\pi is}{2}}\ell_{E,s}(x)-e^{\frac{\pi is}{2}}\ell_{E,s}(1-x) \right),
\end{equation}
the inverse is
\begin{equation}\label{hu-le-zeta-inv}
\ell_{E,s}(x)=\frac{\Gamma(1-s)}{2\pi^{1-s}}\left(e^{\frac{\pi i(1-s)}{2}}\zeta_{E}(1-s,x)-e^{-\frac{\pi i(1-s)}{2}}\zeta_{E}(1-s,1-x)\right).
\end{equation}

The limiting case $x\to1^-$ of (\ref{hu-le-zeta}) leads to the asymmetric form of the functional equation for $\zeta_E(s)$
and $\lambda(s)$
\begin{equation}\label{hu-le-zeta-ft}
\zeta_{E}(1-s)=-\frac{2\Gamma(s)}{\pi^s}\cos\left(\frac{\pi s}{2}\right)\lambda(s),
\end{equation}
where $\lambda(s)$ is the the Dirichlet lambda function (\ref{lam-def}). Similarly, in (\ref{hu-le-zeta-inv}), with $x=\frac12,$ we have
\begin{equation}\label{hu-le-zeta-ft-2}
\ell_{E,s}\left(\frac12\right)=\frac{i\Gamma(1-s)}{\pi^{1-s}}\sin\left(\frac{\pi(1-s)}{2}\right) \zeta_{E}\left(1-s,\frac12\right).
\end{equation}

\begin{theorem}\label{thm-zeta-le}
Let $s,s'\in\mathbb R_0^-.$ Then we have
\begin{equation}\label{thm-z-l}
\int_0^1\zeta_E(s',x)\ell_{E,s}(x)dx=\pi^{s'-1}\Gamma(1-s')e^{\frac{\pi i}{2}(1-s')}\lambda(1+s-s').
\end{equation}
\end{theorem}
\begin{proof}
We have the functional equation for the Hurwitz-type Euler zeta function (see (\ref{hu-le-zeta-inv})):
$$\ell_{E,s}(x)=\frac{i\Gamma(1-s)}{2\pi^{1-s}}\left(e^{-\frac{\pi is}{2}}\zeta_{E}(1-s,x)+e^{\frac{\pi is}{2}}\zeta_{E}(1-s,1-x)\right).$$
From which, we obtain
\begin{equation}\label{int-z-l}
\begin{aligned}
\int_0^1\zeta_E(s',x)\ell_{E,s}(x)dx
&=\frac{i\Gamma(1-s)}{2\pi^{1-s}}\left[e^{-\frac{\pi is}{2}}\int_0^1\zeta_E(s',x)\zeta_{E}(1-s,x)dx \right.\\
&\quad+e^{\frac{\pi is}{2}}\left.\int_0^1\zeta_E(s',x)\zeta_{E}(1-s,1-x)\right].
\end{aligned}
\end{equation}
Substituting (\ref{thm-re}) and (\ref{thm-re-2}) in the integral   (\ref{int-z-l}), we deduce that
\begin{equation}\label{int-z-2}
\begin{aligned}
\int_0^1\zeta_E(s',x)\ell_{E,s}(x)dx
&=\frac{i\Gamma(1-s)\Gamma(s)\Gamma(1-s')}{\pi^{2-s'}}\lambda(1+s-s') \\
&\quad\times\left[e^{-\frac{\pi is}{2}}\cos\left(\frac{\pi(1-s-s')}{2}\right)\right. \\
&\quad\quad+e^{\frac{\pi is}{2}}\left.\cos\left(\frac{\pi(1-s+s')}{2}\right)\right].
\end{aligned}
\end{equation}
Now the factor on the right hand side of (\ref{int-z-2}) is
$$\left[e^{-\frac{\pi is}{2}}\cos\left(\frac{\pi(1-s-s')}{2}\right)
+e^{\frac{\pi is}{2}}\cos\left(\frac{\pi(1-s+s')}{2}\right)\right]=e^{-\frac{\pi is'}{2}}\sin(\pi s).$$
Therefore, from the reflection formula
$$\Gamma(s)\Gamma(1-s)=\frac{\pi}{\sin(\pi s)},$$
we have
$$\int_0^1\zeta_E(s',x)\ell_{E,s}(x)dx=\frac{i\Gamma(1-s')}{\pi^{1-s'}}e^{-\frac{\pi is'}{2}}\lambda(1+s-s').$$
This completes the proof.
\end{proof}

Let $\phi(x,a,s)$ be the power Dirichlet series defined for Re$(s)>1,$ $x$ real, $a\neq$ negative integer or zero, by the series
\begin{equation}\label{di-phi}
\phi(x,a,s)=\sum_{n=0}^\infty\frac{e^{2n\pi ix}}{(n+a)^s}
\end{equation}
(see \cite[p. 161, (1.1)]{Ap0}). It is known that $\phi(x,a,s)$ can be extended to whole $s$-plane by means of the contour integral.
The value of $\phi(x,a,s)$ when $s$ is 0 or a negative integer can be calculated by applying Cauchy's residue theorem.
For $m\in\mathbb N_0,$ the Apostol's formula (see \cite[p. 164)]{Ap0}) gives
\begin{equation}\label{di-phi-sp-v}
\phi(x,a,-m)=-\frac{B_{m+1}(a,e^{2\pi ix})}{m+1},
\end{equation}
where $B_{m}(a,\alpha)$ are the Apostol-Bernoulli polynomials defined by
\begin{equation}\label{ap-be-def}
\frac{ze^{az}}{\alpha e^z-1}=\sum_{m=0}^\infty B_{m}(a,\alpha)\frac{z^m}{m!}.
\end{equation}
We refer to \cite{Ap0,Ba2,HKK,Luof,NFV} for an account of the properties of Apostol-Bernoulli polynomials $B_{m}(a,\alpha).$
For a later purpose,
we observe that, by (\ref{le-eu-zeta}) and (\ref{di-phi}),  the following equality is established,
\begin{equation}\label{phi-ruch}
\ell_{E,s}(x)=2^{-s}e^{\pi ix}\phi\left(x,\frac12,s\right),
\end{equation}
so that by (\ref{di-phi-sp-v}) and (\ref{phi-ruch}),
\begin{equation}\label{phi-ruch-sp-v}
\ell_{E,-m}(x)=-2^{m}e^{\pi ix}\frac{B_{m+1}\left(\frac12,e^{2\pi ix}\right)}{m+1},
\end{equation}
where $m\in\mathbb N_0.$

We observe that the function $\zeta_{E}(s,x)$ in (\ref{HEZ}) is a linear combination of the Hurwitz zeta functions $\zeta(s,x)$
in (\ref{Hurwitz}) when $x$ is a rational number.

\begin{theorem}\label{ile-hu-ft}
Let $p,q$ be integers, $q\in\mathbb N$ and $1\leq p\leq q.$ Then for all $s\in\mathbb C,$  we have
\begin{equation}\label{h-h-ft-e}
\begin{aligned}
\zeta_E\left(1-s,\frac{p}{q}\right)
=\frac{2\Gamma(s)}{(2q\pi)^s}\sum_{r=0}^{q-1}
\cos\left(\frac{\pi s}{2}-\frac{(2r+1)\pi p}{q}\right)\zeta\left(s,\frac{2r+1}{2q}\right).
\end{aligned}
\end{equation}
\end{theorem}

\begin{remark}
This is an analogue of \cite[p. 261, Theorem 12.8]{Ap2}, where the Hurwitz zeta function $\zeta(s,x)$ instead of $\zeta_E(s,x).$
\end{remark}

\begin{proof}[Proof of Theorem \ref{ile-hu-ft}]
Following an idea in \cite[p.~337, (8)]{Sri}, we have
\begin{equation}\label{choi}
\sum_{n=1}^\infty f(n)=\sum_{r=1}^{q}\sum_{k=0}^\infty f(kq+r)
\end{equation}
if the involved series is absolutely convergent.
Applying (\ref{choi}) to the series (\ref{le-eu-zeta}), we have (see \cite[p. 1530, (8b)]{CK} and \cite[p. 1487, (3.4)]{CK10})
\begin{equation}\label{h-h-type}
\ell_{E,s}\left(\frac pq\right)=\frac1{(2q)^s}\sum_{r=1}^{q}e^{(2r-1)\pi i\frac{p}{q}}\zeta\left(s,\frac{2r-1}{2q}\right).
\end{equation}
Therefore, if we take $x=\frac pq$ in the formula (\ref{hu-le-zeta}) to (\ref{h-h-type}), in considerations of $\ell_{E,s}(1-x)=-\ell_{E,s}(x),$
we have the following equality
$$
\begin{aligned}
\zeta_E\left(1-s,\frac{p}{q}\right)
&=\frac{\Gamma(s)}{\pi^s}\left(e^{-\frac{\pi is}{2}}\ell_{E,s}\left(\frac pq\right)+e^{\frac{\pi is}{2}}
\ell_{E,s}\left(-\frac pq\right)\right) \\
&=\frac{2\Gamma(s)}{(2q\pi)^s}\sum_{r=1}^{q}
\cos\left(\frac{\pi s}{2}-\frac{(2r-1)\pi p}{q}\right)\zeta\left(s,\frac{2r-1}{2q}\right),
\end{aligned}
$$
which holds true, by the principle of analytic continuation, for all admissible values of $s\in\mathbb C.$
\end{proof}

\begin{corollary}[{\cite[p. 1529, Theorem B]{CK}}]\label{ile-hu-ft-thm}
Let $p,q$ be integers, $q\in\mathbb N$ and $1\leq p\leq q.$
For $m\in\mathbb N,$ the Euler polynomials $E_m(x)$ at rational arguments are given by
\begin{equation}\label{h-h-ft-odd}
\begin{aligned}
E_{m}\left(\frac{p}{q}\right)
=\frac{4m!}{(2q\pi)^{m+1}}\sum_{r=1}^{q}
\sin\left(\frac{(2r-1)\pi p}{q}-\frac{m\pi }{2}\right)\zeta\left(m+1,\frac{2r-1}{2q}\right).
\end{aligned}
\end{equation}
\end{corollary}
\begin{proof}
By setting $s=m+1$ ($m\in\mathbb N_0$) in (\ref{h-h-ft-e}) and using (\ref{EZ-Eu}), we obtain the formula in (\ref{h-h-ft-odd}),
which completes the proof.
\end{proof}

\begin{corollary}
Let $p,q$ be integers, $q\in\mathbb N$ and $1\leq p< q.$ Then we have
\begin{equation}\label{h-h-type-co}
B_{m+1}\left(\frac 12,e^{2\pi i\frac pq}\right)=q^m\sum_{r=1}^{q}e^{2(r-1)\pi i\frac{p}{q}}B_{m+1}\left(\frac{2r-1}{2q}\right),
\quad m\in\mathbb N_0.
\end{equation}
\end{corollary}
\begin{proof}
Letting $m\in\mathbb N_0,$ then replacing $s$ by $-m$ in (\ref{h-h-type}), using (\ref{H-Sp}) and (\ref{phi-ruch-sp-v}),
we get our result.
\end{proof}
\begin{remark} This identity is also a special case of the multiplication theorem for Apostol-Bernoulli polynomials (see eg. \cite[Eq. (17)]{Luo}).
\end{remark}

Since $\ell_{E,s}(x)$ is a relative of the Hurwitz-type Euler zeta function $\zeta_E(s,x)$
defined (\ref{HEZ}) (see (\ref{hu-le-zeta}) and (\ref{hu-le-zeta-inv}) above),
it is natural to expect a relation involving $\zeta (s,x)$:

\begin{lemma}[Generalized Eisenstein formula]\label{GEF}
Let $p,q$ be integers, $q\in\mathbb N$ and $1\leq p< q.$ Then we have
\begin{equation}\label{inverse-Ei}
\zeta\left(s,\frac{2p-1}{2q}\right)=\frac1{q}\sum_{r=1}^{q}(2q)^se^{-(2p-1)\pi i\frac rq}  \ell_{E,s}\left(\frac rq\right).
\end{equation}
\end{lemma}

\begin{remark}
The equation (\ref{inverse-Ei}) was announced without a proof in \cite[p. 1487, (3.3)]{CK10}.
\end{remark}

\begin{proof}[Proof of Lemma \ref{GEF}]
The formula (\ref{Hurwitz}) could be written in the following form
\begin{equation}~\label{Hurwitz-ft-e}
\begin{aligned}
\zeta\left(s,\frac{2p+1}{2q}\right)&=(2q)^s\sum_{k=0}^\infty\frac{1}{(2(qk+p)+1)^s} \\
&=(2q)^s\sum_{ n=0 \atop n\equiv p\!\!\!\pmod q}^\infty\frac1{(2n+1)^s}.
\end{aligned}
\end{equation}
Note that in fact (see \cite[p. 158, Theorem 8.1]{Ap2})
$$\frac1q\sum_{r=1}^qe^{2\pi i(n-p)\frac rq}
=\begin{cases}
1 & \text{ if } n\equiv p \pmod q\\
0 & \text{ otherwise},\end{cases}
$$
and hence
\begin{equation}~\label{Hurwitz-ft-e}
\begin{aligned}
\zeta\left(s,\frac{2p+1}{2q}\right)
&=(2q)^s\sum_{ n=0 }^\infty\frac{\frac1q\sum_{r=1}^qe^{2\pi i(n-p)\frac rq}}{(2n+1)^s} \\
&=\frac1q \sum_{r=1}^q(2q)^s\sum_{n=0}^\infty\frac{e^{(2n+1)\pi i\frac rq}}{(2n+1)^s} e^{-(2p+1)\pi i\frac rq}
\end{aligned}
\end{equation}
for $1\leq p< q,$ the conclusion is immediate.
\end{proof}

\begin{corollary}
Let $p,q$ be integers, $q\in\mathbb N$ and $1\leq p< q.$ Then we have
\begin{equation}\label{inverse-Ei-co}
B_{m+1}\left(\frac{2p-1}{2q}\right)=\frac1{q^{m+1}}\sum_{r=1}^{q}e^{-2(p-1)\pi i\frac rq} B_{m+1}\left(\frac 12,e^{2\pi i\frac rq}\right),
\quad m\in\mathbb N_0.
\end{equation}
Here, the term $B_{m+1}\left(\frac 12, 1\right)$ is understood as $B_{m+1}\left(\frac 12\right)$
(the value of the Bernoulli polynomials $B_{m+1}(x)$ at $x=\frac12$).
\end{corollary}
\begin{proof}
Let $m\in\mathbb N_0,$ and replace $s$ by $-m$ in (\ref{inverse-Ei}). Then, using (\ref{H-Sp}) and (\ref{phi-ruch-sp-v}),
we get our result.
\end{proof}

\begin{theorem}\label{mul-dis-eq}
Let $p,q$ be integers, $q\in\mathbb N$ and $1\leq p\leq q.$ Then we have the identities
\begin{equation}\label{ei-form-h-l-zeta}
\begin{aligned}
\frac1{q}\sum_{r=1}^{q}(2q)^s&e^{-(2p+1)\pi i\frac rq}  \ell_{E,s}\left(\frac rq\right) \\
&=\frac{\Gamma(1-s)}{\pi^{1-s}}\left\{e^{-\frac{\pi i(1-s)}{2}}\ell_{E,1-s}(x)
-e^{\frac{\pi i(1-s)}{2}}\ell_{E,1-s}(1-x) \right\}
\end{aligned}
\end{equation}
and
\begin{equation}\label{ei-form-e-zeta}
\begin{aligned}
\frac1{(2q)^s}\sum_{r=1}^{q}&e^{(2r-1)\pi i\frac pq}\zeta\left(s,\frac{2r-1}{2q} \right)\\
&=\frac{\Gamma(1-s)}{2\pi^{1-s}}\left\{e^{\frac{\pi i(1-s)}{2}}\zeta_{E}\left(1-s,\frac pq\right)
-e^{-\frac{\pi i(1-s)}{2}}\zeta_{E}\left(1-s,1-\frac pq\right)\right\}.
\end{aligned}
\end{equation}
\end{theorem}
\begin{proof}
From (\ref{hu-le-zeta}), (\ref{hu-le-zeta-inv}), (\ref{h-h-type}) and (\ref{inverse-Ei}), we prove the theorem.
\end{proof}

We will use $[x]$ to denote the integer part of $x\in\mathbb R$ (i.e., $[x]$ is the largest integer $\leq x$);
then, the fractional part of $x$ will be $\{x\}=x-[x].$

\begin{lemma}\label{exp-zeta-lem}
For $n\in\mathbb N$ we have
\begin{equation}\label{four-euler-com}
E_{n}(\{x\})=2(-i)^{n-1}n!\sum_{k=0}^\infty\frac{(-1)^{n-1}e^{(2k+1)\pi ix}+e^{-(2k+1)\pi ix}}{((2k+1)\pi)^{n+1}}.
\end{equation}
\end{lemma}
\begin{proof}
Let $\{x\}=x-[x].$
The function ${E}_n(\{x\})$ is a quasi-periodic function of period 2 (see \cite{HKK}).
The Fourier series of ${E}_n(\{x\})$ are given as follows (\cite[p. 805, 23.1.16]{AS}):
\begin{equation}\label{four-euler}
{E}_n(\{x\})=\frac{4n!}{\pi^{n+1}}\sum_{k=0}^\infty\frac{\sin((2k+1)\pi x-\frac12\pi n)}{(2k+1)^{n+1}},
\end{equation}
where $0\leq x<1$ if $n\in\mathbb N$ and $0<x<1$ if $n=0,$ immediately yields
\begin{equation}\label{four-euler-even}
{E}_{2m}(\{x\})=4(-1)^m(2m)!\sum_{k=0}^\infty\frac{\sin((2k+1)\pi x)}{((2k+1)\pi)^{2m+1}},
\end{equation}
and
\begin{equation}\label{four-euler-odd}
{E}_{2m+1}(\{x\})=4(-1)^{m-1}(2m+1)!\sum_{k=0}^\infty\frac{\cos((2k+1)\pi x)}{((2k+1)\pi)^{2m+2}}.
\end{equation}
Thus (\ref{four-euler-com}) follows immediately from (\ref{four-euler-even}) and (\ref{four-euler-odd}).
\end{proof}

\def\r{\alpha}

\begin{theorem}\label{exp-zeta-thm}
Let $m>1$ be a fixed positive odd integer and let $\r$ be an integer such that $\r\not\equiv0\pmod{m}.$ Then, for $n\in\mathbb N,$
we have
\begin{equation}\label{exp-zeta-eq}
\begin{aligned}
\sum_{r=1}^{m-1}(-1)^re^{-2\pi ir\frac \r m}\ell_{E,1-n}\left(\frac rm\right)
&=\frac{(-1)^{n-1}}{4}\left(m^nE_{n-1}\left(\frac{2\r}{m}-\left[\frac{2\r}m\right]\right)+E_{n-1}(0)\right)\\
&\quad-\frac1{2n}\left(m^nB_n\left(\frac{2\r}{m}-\left[\frac{2\r}m\right]\right)+B_n(0)\right).
\end{aligned}
\end{equation}
\end{theorem}

\begin{proof}
It follows, from (\ref{hu-le-zeta-inv}), that for odd $m>1,$
\begin{equation}\label{pf-1}
\begin{aligned}
\sum_{r=1}^{m-1}(-1)^r&e^{-2\pi ir\frac \r m}\ell_{E,1-n}\left(\frac rm\right) \\
&=\frac{(n-1)!i^n}{2\pi^{n}}\left(\sum_{r=1}^{m-1}(-1)^re^{-2\pi ir\frac \r m}\zeta_{E}\left(n,\frac rm\right)\right. \\
&~\quad\quad\quad\quad\quad\quad\left.-(-1)^n\sum_{r=1}^{m-1}(-1)^re^{-2\pi ir\frac \r m}\zeta_{E}\left(n,\frac{m-r}m\right)\right) \\
&=\frac{(n-1)!i^n}{2\pi^{n}}\left(\sum_{r=1}^{m-1}(-1)^re^{-2\pi ir\frac \r m}\zeta_{E}\left(n,\frac rm\right)\right. \\
&~\quad\quad\quad\quad\quad\quad\left.+(-1)^n\sum_{r'=1}^{m-1}(-1)^{r'}e^{-2\pi ir'\frac \r m}\zeta_{E}\left(n,\frac{r'}m\right)\right)
\\
&=\frac{(n-1)!i^n}{2\pi^{n}}\sum_{r=1}^{m-1}(-1)^r\sum_{k=0}^\infty
\frac{(-1)^k\left[e^{-2\pi ir\frac \r m}+(-1)^ne^{2\pi ir\frac \r m}\right]}{\left(k+\frac \r m\right)^n}
 \\
&=\frac{(n-1)!i^nm^n}{2\pi^{n}}\sum_{r=1}^{m-1}\sum_{k=0}^\infty
\frac{(-1)^{mk+r}\left[e^{-2\pi i\r\frac {(mk+r)} m}+(-1)^ne^{2\pi i\r\frac  {(mk+r)} m}\right]}{\left(mk+r\right)^n} \\
&=\frac{(n-1)!i^nm^n}{2\pi^{n}}\sum_{{h=1\atop h\not\equiv 0\!\!\!\pmod{m}}}^\infty
\sum_{h=1}^\infty
\frac{(-1)^{h}\left[e^{-2\pi i\r\frac {h} m}+(-1)^ne^{2\pi i\r\frac  {h} m}\right]}{h^n}
\\
&=\frac{(n-1)!i^nm^n}{2\pi^{n}}\left(
\sum_{h=1}^\infty \frac{(-1)^{h}\left[e^{-2\pi i\r\frac {h} m}+(-1)^ne^{2\pi i\r\frac  {h} m}\right]}{h^n} \right.\\
&\quad\quad\quad\quad\quad\quad\quad\quad\left.+\sum_{k=1}^\infty \frac{(-1)^{mk}\left[e^{-2\pi ik\r}+(-1)^ne^{2\pi ik\r}\right]}{(km)^n}
 \right).
\end{aligned}
\end{equation}
In the following, we will use the identities below:
\begin{equation}\label{known-id}
\sum_{n=1}^\infty\frac{(-1)^n}{n^s}=-\sum_{k=1}^\infty\frac{1}{(2k-1)^s}+\sum_{k=1}^\infty\frac{1}{(2k)^s}
\end{equation}
and
\begin{equation}\label{known-id-2}
B_n(\{x\})=-(-i)^nn!\sum_{k=1}^\infty \frac{e^{2\pi ikx}+(-1)^ne^{-2\pi ikx}}{(2\pi k)^n}
\end{equation}
for $n\in\mathbb N$ (see \cite[p. 13, \S3]{Wa} or \cite{Li10}).
Note that, by (\ref{four-euler-com}) and $E'_n(\{x\})=nE_{n-1}(\{x\}),$ we have
\begin{equation}\label{known-id-3}
E'_n(\{x\})=2(-i)^nn!\sum_{k=0}^\infty \frac{(-1)^ne^{(2k+1)\pi ix}+e^{-(2k+1)\pi ix}}{\pi^n(2k+1)^n}.
\end{equation}
Therefore we have
\begin{equation}\label{pf-2}
\begin{aligned}
\sum_{r=1}^{m-1}(-1)^r&e^{-2\pi i\frac rm}\ell_{E,1-n}\left(\frac rm\right) \\
&=\frac{(n-1)!i^nm^n}{2\pi^{n}}\left(
\sum_{k=1}^\infty\frac{e^{-2\pi i(2k)\frac \r m}+(-1)^ne^{2\pi i(2k)\frac \r m}}{(2k)^n}\right. \\
&\qquad\qquad\qquad\quad\quad -\left.\sum_{k=1}^\infty\frac{e^{-2\pi i(2k-1)\frac \r m}+(-1)^ne^{2\pi i(2k-1)\frac \r m}}{(2k-1)^n} \right. \\
&\qquad\qquad\qquad\quad\quad +\left. m^{-n}\sum_{k=1}^\infty\frac{e^{-2\pi i(2k)\r}+(-1)^ne^{2\pi i(2k)\r}}{(2k)^n}\right. \\
&\qquad\qquad\qquad\quad\quad -\left. m^{-n}\sum_{k=1}^\infty\frac{e^{-2\pi i(2k-1)\r}+(-1)^ne^{2\pi i(2k-1)\r}}{(2k-1)^n}
\right)
\\
&=\frac{(-1)^{n-1}}{4n}\left(m^nE'_n\left(\frac{2\r}{m}-\left[\frac{2\r}m\right]\right)+E'_n(2\r-[2\r])\right)\\
&\quad-\frac1{2n}\left(m^nB_n\left(\frac{2\r}{m}-\left[\frac{2\r}m\right]\right)+B_n(2\r-[2\r])\right) \\
&=\frac{(-1)^{n-1}}{4}\left(m^nE_{n-1}\left(\frac{2\r}{m}-\left[\frac{2\r}m\right]\right)+E_{n-1}(0)\right)\\
&\quad-\frac1{2n}\left(m^nB_n\left(\frac{2\r}{m}-\left[\frac{2\r}m\right]\right)+B_n(0)\right).
\end{aligned}
\end{equation}
This completes our proof.
\end{proof}

\begin{remark}
Equation (\ref{exp-zeta-eq}) is an analogue of Theorem D in \cite{Wa}, which is an extension of the classical
Eisenstein formula
$$
\frac\alpha m-\left[\frac\alpha m\right]-\frac12=-\frac1{2m}\sum_{\gamma=1}^{m-1}
\sin\left(\frac{2\pi \gamma\alpha}{m}\right)\cot\left(\frac{\pi\gamma}{n}\right),
$$
where $\alpha$ is an integer with $\alpha\not\equiv0\pmod m$ and $m>2$ is a fixed positive integer.
\end{remark}

\section*{Acknowledgment} We thank the referee for his/her helpful comments and suggestions.

\bibliography{central}

\end{document}